\documentclass{amsart}

\usepackage{idempotent}

\newcommand{\Per}{\mathrm{Per}}

\newtheorem*{theoremA}{Theorem A}
\newtheorem*{theoremB}{Theorem B}
\newtheorem*{theoremC}{Theorem C}
\newtheorem*{theoremD}{Theorem D}

\newcommand{\subgrp}{\le}

\title[Alternate proof]{An alternate proof of idempotent relations
  among periodic points and quotients}

\author{Xander Faber}
\author{Michelle Manes}
\author{Laura Walton}
\thanks{MM  partially supported by
Simons Collaboration Grant \#359721.}

\begin{document}
\begin{abstract}
We give a short proof of an idempotent relation formula for counting
periodic points of endomorphisms defined over finite fields. The
original proof of this result, due to Walton, uses formal manipulation
of arithmetic zeta functions, whereas we deduce the result directly
from a related theorem of Kani and Rosen.
\end{abstract}
\maketitle


\section{Introduction}

Let $X$ be a scheme of finite type over a finite field $\FF_q$, and
let $G \subgrp \Aut_X(\FF_q)$ be a finite subgroup of the automorphism
group of $X$, each element of which is defined over $\FF_q$. For each
subgroup $H \subgrp G$, we write $\pi_H \colon X \to X/H$ for the
quotient of $X$ by the action of $H$. (See, e.g., \cite[p.65]{MumfordAV}.)
 
Given a subgroup $H \subgrp G$, we define the associated
  idempotent in the group ring of $G$ to be
\[
\epsilon_H = \frac{1}{|H|}\sum_{h \in H} h \in \QQ[G].
\]
For a collection of integers $n_H$, we write $\sum_{H \subgrp G} n_H
\epsilon_H \sim 0$ if the sum is killed by every rational character of
$G$. Kani and Rosen proved a curious relation among the numbers of
points on $X$ and its various quotients:

\begin{theoremA}[\cite{kani_rosen_idempotents_jnt}]
  \label{Thm:KR}
    Let $X$ be a scheme of finite type over $\FF_q$, and let $G
    \subgrp \Aut_X(\FF_q)$ be a finite group of automorphisms. If
    $\sum_{H \subgrp G} n_H \epsilon_H \sim 0$, then
    \[
    \sum_{H \subgrp G} n_H \left|(X/H)(\FF_q)\right| = 0.
    \]
\end{theoremA}

The third author extended this result to count the \textit{periodic
  points} of an endomorphism defined over $\FF_q$.  More precisely,
let $V$\footnote{Both $X$ and $V$ represent finite-type schemes over a
  finite field. Throughout, we use the letter $X$ in the context of
  algebraic geometry and $V$ in a dynamical context. } be a scheme of
finite type over $\FF_q$ and let $f \colon V \to V$ be an endomorphism
of $V$. Define the automorphism group of $f$ by $\Aut_f(\FF_q) = \{g \in \Aut_V(\FF_q) : gfg^{-1} = f\}$, and
suppose that $G $ is a finite subgroup of $\Aut_f(\FF_q)$.  For a subgroup $H \subgrp G$, we write $f_H \colon V/H
\to V/H$ for the unique $\FF_q$-endomorphism that makes the following
diagram commute:
\[
\begin{tikzcd}
    V \ar[r, "f"] \ar[d, "\pi_H"]& V \ar[d, "\pi_H"] \\
    V/H \ar[r, "f_H"] &V/H
\end{tikzcd}
\]
Write $\Per(V/H,f_H)(\FF_q)$ for the set of $\FF_q$-rational periodic
points for $f_H$.

\begin{theoremB}[\cite{Walton_idempotent}]
  \label{Thm:Walton}
    With the above setup, if $\sum_{H \subgrp G} n_H \epsilon_H \sim 0$, then 
    \[
    \sum_{H \subgrp G} n_H \left| \Per(V/H,f_H)(\FF_q) \right| = 0.
    \]
\end{theoremB}

The third author's technique is to define zeta and $L$-functions that capture
the arithmetic of periodic points, and then to formally manipulate these objects
to arrive at an idempotent relation for zeta functions. (See Theorem~D in the
final section.) Later, we observed that Theorems~A and~B can be deduced from
each other.  One obtains Theorem~A from Theorem~B by setting $V = X$ and $f =
\id$. The reverse implication is not immediately obvious though: there is no
scheme that parameterizes the set of \textit{all} periodic points for an
endomorphism $f \colon V \to V$. Nevertheless, finite fields are finite (!), so
we are able to find a scheme that parameterizes enough of the periodic points to
arrive at the desired conclusion. We execute this plan this in the next section.



\section{Counting periodic points for quotients}

We keep the setup from the previous section. 

\begin{lemma}
  \label{lem:period}
  Define 
\[
M = \max\left(|(V/H)(\FF_q)|! \ : \ H \subgrp G\right).
\]
(Note the factorial.) Fix $H \subgrp G$, and let $P \in (V/H)(\FF_q)$. If
$P$ is periodic for $f_H$, then $P$ has period dividing $M$.
\end{lemma}

\begin{proof}
Observe that the orbit of $P$ consists of $\FF_q$-rational points of
$V/H$. So
  \[
  \left|\{f^n_H(P) \ : \ n \geq 0\}\right| \leq \left|(V/H)(\FF_q)\right| \leq
  \max\left(|(V/H)(\FF_q)| \ : \ H \subgrp G\right).
  \]
So the length of the orbit of $P$ divides the factorial of the
quantity on the right, which is $M$.
\end{proof}

Now define
\[
N = |G| \cdot M = |G| \cdot \max\left(|(V/H)(\FF_q)|! \ : \ H \subgrp G\right).
\]
Write $\Per_N(V, f)$ for the closed subscheme of $V$ consisting of
points of period dividing $N$. It is defined by the following
cartesian square:
\[
\begin{tikzcd}
  \Per_N(V,f) \ar[r,"i"] \ar[d, "j"]  & V \ar[d, "f^N"] \\
  V \ar[r,"\id"] & V
\end{tikzcd}
\]

\begin{proposition}
The closed subscheme  $\Per_N(V,f) \subset V$ is $G$-invariant.
\end{proposition}

\begin{proof}
Let $g \in G$. The action of $g$ on $\Per_N(V,f)$ is given by the composition 
$g \circ i$, and we would like this composition to factor as $i \circ \tilde g$
for some automorphism $\tilde g$ of $\Per_N(V,f)$. The cartesian square defining 
$\Per_N(V,f)$ and the fact that $f^N$ commutes with $G$ shows that 
\[
\id \circ g \circ j = g \circ j = g \circ \id \circ j = g \circ f^N \circ i = f^N \circ g \circ i.
\]
Hence, we get the following commutative diagram:
\[
\begin{tikzcd}
\Per_N(V,f) \ar[rrd,bend left=30, "g \circ i"] \ar[ddr,bend right=30, "g \circ j"] \ar[dr,dashed,"\tilde g ?"] & & \\
& \Per_N(V,f) \ar[r,"i"] \ar[d, "j"]  & V \ar[d, "f^N"] \\
&  V \ar[r,"\id"] & V
\end{tikzcd}
\]
The arrow $\tilde g$ exists and is unique by the universal property of 
$\Per_N(V,f)$ as a fiber product. 

To see that $\tilde g$ is an isomorphism, run the same argument two more times
with $g$ replaced by $g^{-1}$ and by the composition $g \circ g^{-1}$.
\end{proof}


For $H \subgrp G$, we define $\Per_N(V/H,f_H)$ by the following cartesian square:
\[
\begin{tikzcd}
  \Per_N(V/H,f_H) \ar[r] \ar[d]  & V/H \ar[d, "f_H^N"] \\
  V/H \ar[r,"\id"] & V/H
\end{tikzcd}
\]
The $H$-quotient of the locus of periodic points for $f: V \to V$ is subtly 
different from the locus of periodic points for $f_H : V/H \to V/H$. As we 
need to compare these sets, define a map
\[
i_H \colon \left(\Per_N(V,f)/H\right)(\FF_q) \to \Per_N(V/H,f_H)(\FF_q)
\]
as follows. The $\FF_q$-rational points of $\Per_N(V,f)/H$ correspond to Galois
invariant $H$-orbits of $\bar \FF_q$-rational points of $\Per_N(V,f)$. Let $Q
\in \left(\Per_N(V,f)/H\right)(\FF_q)$, and let $P \in \Per_N(V,f)(\bar \FF_q)$
be a representative of the corresponding $H$-orbit. Set $i_H(Q) = \pi_H(P)$. The
map $i_H$ is well-defined because any other representative $P'$ for the
$H$-orbit of $P$ has the same image under $\pi_H$. Since the $H$-orbit of $P$ is
$\Gal(\bar \FF_q / \FF_q)$-invariant, we see that $\pi_H(P) \in
(V/H)(\FF_q)$. Moreover, since $P$ is periodic of period dividing $N$, so is
$\pi_H(P)$.

\begin{lemma}
  \label{lem:bijection}
  $i_H$ is a bijection.
\end{lemma}

\begin{proof}
First, injectivity. If $Q,Q'$ are elements of
$\left(\Per_N(V,f)/H\right)(\FF_q)$ such that $i_H(Q) = i_H(Q')$, then
the corresponding lifts $P,P' \in \Per_N(V,f)(\bar \FF_q)$ lie in the
same $H$-orbit. So there is $h \in H$ such that $P' = h(P)$. But this
means $Q = Q'$.

Next, surjectivity. Take $R \in \Per_N(V/H,f_H)(\FF_q)$. Lift it to $P
\in V(\bar \FF_q)$ such that $\pi_H(P) = R$. Now $R$ is $f_H$-periodic
and $\FF_q$-rational, so its period divides $M$ by
Lemma~\ref{lem:period}. It follows that there is $h \in H$ such that $f^M(P) =
h(P)$. If we let $d$ be the order of $h$, then the fact that $h$ commutes with
$f$ shows that
\[
f^{Md}(P) = \underbrace{f^M \circ \cdots \circ f^M}_{d \text{ times}}(P)
= \underbrace{h \circ \cdots \circ h}_{d \text{ times}}(P) = P.
\]
Hence, $P$ is $f$-periodic of period dividing $N = M \cdot |G|$, and $P \in
\Per_N(V,f)(\bar \FF_q)$. The $H$-orbit of $P$ is Galois invariant (as it's a
lift of an $\FF_q$-point under $\pi_H$), so its image in $\Per_N(V,f)/H$ is an
$\FF_q$-rational point $Q$. By construction $i_H(Q) = R$, as desired.
\end{proof}

\begin{remark}
Only the surjectivity argument required our careful choice of $N$.
\end{remark}

\begin{proposition}
  \label{Prop:equalities}
  For any $H \subgrp G$,
  \[
  \left|\left(\Per_N(V,f)/H\right)(\FF_q)\right|
  = \left|\Per_N(V/H,f_H)(\FF_q)\right|
  = \left|\Per(V/H,f_H)(\FF_q)\right|.
  \]
\end{proposition}

\begin{proof}
  The first equality is immediate from Lemma~\ref{lem:bijection}. 
  For the second, it's clear that we have an inclusion
  \[
  \Per_N(V/H,f_H)(\FF_q) \subseteq \Per(V/H,f_H)(\FF_q).
  \]
  If $P \in (V/H)(\FF_q)$ is periodic for $f_H$, then its period
  divides $M$ by Lemma~\ref{lem:period}. So its period divides $N = M
  \cdot |G|$. Hence $P \in \Per_N(V/H,f_H)(\FF_q)$, and the reverse
  inclusion holds.
\end{proof}

\begin{proof}[Proof of Theorem~B]
  Since $\Per_N(V,f)$ is defined over $\FF_q$ and $G$-invariant, we
  can apply Theorem~A to obtain
  \[
  \sum_{H \subgrp G} n_H \left|\left(\Per_N(V,f)/H\right)(\FF_q)\right| = 0.
  \]
  Now apply the preceding proposition.
\end{proof}


\section{Zeta functions for periodic points and quotients}

For a scheme $X/\FF_q$ of finite type, write
\[
\zeta(X,s) = \exp\left(\sum_{n \geq 1} |X(\FF_{q^n})| \frac{q^{-ns}}{n}\right)
\]
for the usual zeta function of $X$. 

\begin{theoremC}[\cite{kani_rosen_idempotents_jnt}]
  \label{Thm:KR_zeta}
    Let $X$ be a scheme of finite type over $\FF_q$, and let $G
    \subgrp \Aut_X(\FF_q)$ be a finite group of automorphisms. If
    $\sum_{H \subgrp G} n_H \epsilon_H \sim 0$, then
    \[
    \prod_{H \subgrp G} \zeta(X/H,s)^{n_H} = 1.
    \]
\end{theoremC}

Switching to a dynamical setting, define the arithmetic zeta function of the pair $(V,f)$ to be
\[
\zeta_f(V,s) = \exp\left(\sum_{n \geq 1} |\Per(V,f)(\FF_{q^n})| \frac{q^{-ns}}{n}\right).
\]
(This is not the same as
the Artin/Mazur zeta function described in \cite{Artin_Mazur}). By the Weil bounds,
the series is convergent for $\Re(s) > \dim(V)$.

\begin{theoremD}[\cite{Walton_idempotent}]
  \label{Thm:Walton_zeta}
    Let $V$ be a scheme of finite type over $\FF_q$, let $f \colon V
    \to V$ be an endomorphism of $V$, and let $G \subgrp
    \Aut_f(\FF_q)$ be a finite group of automorphisms of $V$ that
    commute with $f$. If $\sum_{H \subgrp G} n_H \epsilon_H \sim 0$, then
  \[
  \prod_{H \subgrp G} \zeta_{f_H}(V/H,s)^{n_H} = 1.
  \]
\end{theoremD}

Taking the log of both sides and looking at the coefficient on
$q^{-s}$ allows one to deduce Theorems~A and~B from Theorems~C
and~D. In fact, the latter two theorems are equivalent to the former
two: apply Theorems~A and~B over the field $\FF_{q^n}$ for varying $n$
and then use that to compare coefficients in Theorems~C and~D.

\subsection*{Acknowledgments}
This material is based upon work supported by and while the second author served at the National Science Foundation. Any
opinion, findings, and conclusions or recommendations expressed in this material are those of the authors
and do not necessarily reflect the views of the National Science Foundation.

The authors thank the anonymous reviewer for helpful comments.

\bibliographystyle{plain}

\end{document}